\documentclass[a4paper,11pt,reqno]{amsart}
\usepackage{amssymb, amsmath, amsfonts, amscd}
\usepackage{mathrsfs, latexsym}
\usepackage[all,ps,cmtip]{xy}
\usepackage{array, longtable}
\usepackage{bm, enumitem}
\usepackage{xcolor}
\usepackage{comment}

\usepackage{geometry}
\usepackage{todonotes}
\title[{Birational boundedness of RC logCY pairs}]{Birational boundedness of rationally connected log Calabi--Yau pairs with fixed index}
\date{\today}

\author{Jingjun Han}
\address{\rm Shanghai Center for Mathematical Sciences, Fudan University, Shanghai 200438, China}
\email{hanjingjun@fudan.edu.cn}
\author{Chen Jiang}
\address{\rm Shanghai Center for Mathematical Sciences \& School of Mathematical Sciences, Fudan University, Shanghai 200438, China}
\email{chenjiang@fudan.edu.cn}

\newcommand{\bQ}{{\mathbb Q}}

\newcommand{\Center}{\operatorname{Center}}

\newcommand{\Zz}{\mathbb{Z}}

\newcommand{\fR}{\mathfrak{R}}

\newcommand{\Rr}{\mathbb{R}}

\newcommand{\cN}{\mathcal{N}}

\newcommand{\Supp}{\operatorname{Supp}}

\newcommand{\mult}{\operatorname{mult}}

\newcommand{\bM}{\mathbf{M}}
\newcommand{\bN}{\mathbf{N}}

\newcommand{\Qq}{\mathbb{Q}}

\newcommand{\lf}{\lfloor}
\newcommand{\rf}{\rfloor}

\newtheorem{thm}{Theorem}[section]
\newtheorem{lem}[thm]{Lemma}
\newtheorem{cor}[thm]{Corollary}
\newtheorem{prop}[thm]{Proposition}

\newtheorem{conj}[thm]{Conjecture}

\theoremstyle{definition}
\newtheorem{defn}[thm]{Definition}

\newtheorem{rem}[thm]{Remark}

\theoremstyle{remark}

\begin{document}
\begin{abstract} 
We show that the set of rationally connected projective varieties $X$ of a fixed dimension such that $(X,B)$ is klt, and $-l(K_X+B)$ is Cartier and nef for some fixed positive integer $l$, is bounded modulo flops. 
\end{abstract}
\maketitle

%%%%%%%%%
%\pagestyle{myheadings}
%\markboth{\hfill \hfill}{\hfill \hfill}
\numberwithin{equation}{section}
%%%%%%%%%%%%
%\tableofcontents

 \section{Introduction}
 Throughout this paper, we work over an uncountable algebraically closed field of characteristic $0$, for instance, the complex number field $\mathbb{C}$.

A normal projective variety $X$ is a {\it Fano} (resp. {\it Calabi--Yau}) variety if $-K_X$ is ample (resp. $K_X\sim_{\mathbb{Q}} 0$). According to the Minimal Model Program, Fano varieties and Calabi--Yau varieties form fundamental classes in birational geometry as building blocks of algebraic varieties. Hence, it is interesting to ask whether such kinds of varieties satisfy any finiteness or boundedness properties. For Fano varieties, Birkar \cite{Bir19, Bir21} showed that the set of $\epsilon$-lc Fano varieties of a fixed dimension forms a bounded family for a fixed $\epsilon>0$, which is known as the Borisov--Alexeev--Borisov (BAB) Conjecture. For (log) Calabi--Yau varieties, things get more complicated. It is expected that certain Calabi--Yau varieties with special geometric properties, for example, Calabi--Yau manifolds or rationally connected Calabi--Yau varieties, form a bounded family, see \cite{DCS21, CDCHJS21, Jia21, BDCS20, FHS21, Jiao22} for some recent progress. 
 
 %, namely, varieties with nef anti-(log-)canonical divisors
 In this paper, we focus on the following conjecture of M\textsuperscript{c}Kernan and Prokhorov. It includes rationally connected Calabi--Yau varieties, and is a natural generalization of the BAB conjecture as any klt Fano variety is rationally connected \cite{Zha06,HM07}. 
\begin{conj}[{\cite[Conjecture~3.9]{MP04}}]\label{conj: MP} Fix a positive
integer $d$ and a positive real number $\epsilon$. Then the set of varieties $X$ such that
\begin{enumerate}
 \item $X$ is a normal projective variety of dimension $d$,
 \item $X$ is rationally connected,
\item $(X, B)$ is $\epsilon$-lc for some  $\mathbb{R}$-divisor $B\geq 0$, and
\item $-(K_X+B)$ is nef,
\end{enumerate}
is bounded.
\end{conj}

Conjecture~\ref{conj: MP} was proved in dimension $2$ by \cite[Theorem~6.9]{Ale94}. When $\dim X=3$, \cite[Theorem~1.6]{BDCS20} showed that $X$ is bounded modulo flops (see \S\ref{sec: boundedness} for definition). Indeed this holds for any non-canonical klt pair $(X,B)$ when the coefficients of $B$ belong to a set satisfying the descending chain condition, without assuming $X$ being rationally connected, see \cite[Theorem 1.9]{HLL22}, \cite[Theorem 1.6]{Jia21}. Even the case $K_X\equiv 0$ and $B=0$ is sufficiently interesting and widely open in dimension at least $4$ \cite[Conjecture~1.3]{CDCHJS21}.

In this paper, we prove that $X$ is bounded modulo flops in Conjecture~\ref{conj: MP} under an additional assumption that the Cartier index of $K_X+B$ is fixed. The following is our main theorem.

 \begin{thm}\label{thm: main pair rc modulo flop}
 Fix positive integers $d$ and $l$. The set of varieties $X$ such that
 \begin{enumerate}
 \item $X$ is a normal projective variety of dimension $d$,
 \item $X$ is rationally connected,
 \item $(X, B)$ is klt for some $\mathbb{Q}$-divisor $B\geq 0$, and
 \item $-l(K_X+B)$ is Cartier and nef,
 \end{enumerate}
 is bounded modulo flops.
 \end{thm}

In particular, when $l(K_X+B)\sim 0$, Theorem~\ref{thm: main pair rc modulo flop} gives a simpler proof for \cite[Theorem~1.4]{BDCS20}. This special case is already interesting enough as it implies the birational boundedness of elliptic Calabi--Yau varieties with a rational section \cite[Theorems~1.2 and 1.3]{BDCS20}. 

Both the proof of \cite[Theorem~1.4]{BDCS20} and ours heavily rely on Birkar's earlier work on log Calabi--Yau fibrations \cite{Bir18} to finish the induction on the dimension. The difference is that \cite{BDCS20} applies towers of Fano fibrations introduced in \cite{DCS21} which makes the proof quite involved, and the main new ingredient of ours is to use generalized pairs (see \S\ref{sec: gpair}) introduced in \cite{BZ16} to reduce the problem to log Calabi--Yau g-pairs, and then to apply the theory of complements to control the index of the moduli part of the canonical bundle formula. We prove the following theorem under this general setting which fits well with dimension induction.
 % and we can employ the help of recent results of Birkar \cite{Bir18} the canonical bundle formula and
 
 \begin{thm}\label{thm main}
 Fix positive integers $d$ and $l$. The set of varieties $X$ such that
 \begin{enumerate}
 \item $X$ is a normal projective variety of dimension $d$,
 \item $X$ is rationally connected,
 \item $(X, B+{\bf M})$ is a klt g-pair for some  $\mathbb{Q}$-divisor $B\geq 0$ and b-$\mathbb{Q}$-divisor ${\bf M}$,
 \item $l(K_X+B+{\bf M}_X)\sim 0$, and
 \item $l{\bf M}$ is b-Cartier,
 \end{enumerate}
 is bounded modulo flops.
 \end{thm}
 
 In fact, we can more generally prove the following relative version of Theorem~\ref{thm main}. See \S 2 for definitions.
 \begin{thm}\label{thm main rel}
 Fix positive integers $d, l$ and $r$. The set of base-polarized fibrations $\pi: X\to (Z, A)$ such that
 \begin{enumerate}
 \item $X$ is a normal projective variety of dimension $d$,
 \item $\pi:X\to Z$ is a rationally connected fibration,
 \item $(X, B+{\bf M})$ is a klt g-pair for some  $\mathbb{Q}$-divisor $B\geq 0$ and b-$\mathbb{Q}$-divisor ${\bf M}$,
 \item $l(K_X+B+{\bf M}_X)\sim \pi^*L$ for some Cartier divisor $L$ on $Z$,
 \item $l{\bf M}$ is b-Cartier,
 \item $A\ge0 $ is a very ample divisor on $Z$ such that $A^{\dim Z}\leq r$, and
 \item $A-\frac{1}{l}L$ is ample,
 
 \end{enumerate}
 is bounded in codimension one. In particular, such $X$ is bounded in codimension one.
 \end{thm}
 Here if $\dim Z=0$, then we always take $L=A=A^{\dim Z}=0$ by default.

If $-K_X$ is big over $Z$, then Theorem~\ref{thm main rel} is a special case of {\cite[Theorem~2.2]{Bir18}} (see Theorem~\ref{thm: bir18 special}). So here in Theorem~\ref{thm main rel} we treat the general setting that $X\to Z$ is rationally connected, and we expect that even if the torsion index of $K_X+B+{\bf M}_X$ over $Z$ and the b-Cartier index of ${\bf M}$ are not assumed to be bounded, $X\to Z$ should still belong to a bounded family. Conjecture \ref{conj: rel MP} is a generalization of \cite[Theorems~1.2, 2.2]{Bir18}, which could be regarded as the relative version of Conjecture \ref{conj: MP}.
 
 \begin{conj}\label{conj: rel MP}
 Let $d, r$ be positive integers, and $\epsilon$ a positive real number. Then the set of base-polarized fibrations $X\to (Z, A)$ such that there exists a $(d, r, \epsilon)$-logCY rationally connected fibration $(X, B+{\bf M})\to (Z, A)$, is bounded. 
 \end{conj}
 %all the $X$ which belong to the set of $(d,r,\epsilon)$-logCY rationally connected fibrations (see Definition \ref{defn cy fibration}) form a bounded family. 

\noindent\textbf{Acknowledgments.} 
This work was supported by National Key Research and Development Program of China \#2023YFA1010600, NSFC for Innovative Research Groups \#12121001, and National Key Research and Development Program of China \#2020YFA0713200. We would like to thank Caucher Birkar, Guodu Chen, Christopher D. Hacon, Junpeng Jiao, Jihao Liu, and Yujie Luo for their interests and helpful comments.
The authors are members of LMNS, Fudan University.
We thank the referee for useful suggestions and comments. 
 
 \section{Preliminaries}

 A {\it contraction} or a {\it fibration} is a projective morphism $\pi: X\to Z$ between normal varieties such that $\pi_*\mathcal{O}_X = \mathcal{O}_Z$. In this case, for a closed point $z\in Z$, the fiber over $z$ is always denoted by $X_z$.
 
 A {\it base-polarized fibration} $\pi: X\to (Z, A)$
 consists of a fibration $\pi: X\to Z$ and a very ample divisor $A\ge0$ on $Z$.

 A {\it birational contraction} is a birational map $\phi: X\dashrightarrow X'$ between normal varieties such that $\phi_*^{-1}$ does not contract any divisors.
 A {\it small} birational map is a birational map $\phi: X\dashrightarrow X'$ between normal varieties which is isomorphic in codimension one.

 A projective variety is said to be {\it rationally connected} if any two general points can be connected by the image of a rational curve. A fibration is said to be {\it rationally connected} if its general fibers are rationally connected.
 
\subsection{Divisors and b-divisors}
Let $\mathbb{K}$ be either the rational number field $\Qq$ or the real number field $\Rr$.
 Let $X$ be a normal variety. A {\it $\mathbb{K}$-divisor} is a finite $\mathbb{K}$-linear combination $D=\sum_{i} d_{i} D_{i}$ of prime Weil divisors $D_{i}$, and $d_{i}$ denotes the {\it coefficient} of $D_i$ in $D$. A {\it $\mathbb{K}$-Cartier divisor} is a $\mathbb{K}$-linear combination of Cartier divisors. 

% We use $\sim_{\mathbb{K}}$ to denote the $\mathbb{K}$-linear equivalence between $\mathbb{K}$-divisors. For a projective morphism $X\to Z$, we use $\sim_{\mathbb{K},Z}$ to denote the relative $\mathbb{K}$-linear equivalence and use $\equiv_{Z}$ to denote the relative numerical equivalence. 

\begin{defn}[cf. \cite{PS09}]\label{defn bdiv}
Let $X$ be a normal variety. Consider an infinite linear combination $\mathbf{D}:=\sum_D d_D D$, where $d_D\in\mathbb{K}$ and the infinite sum runs over all divisorial valuations of the function field of $X$. For any birational model $Y$ of $X$, the \emph{trace} of $\mathbf{D}$ on $Y$ is defined by $\mathbf{D}_Y:=\sum_{\mathrm{codim}_Y D=1} d_D D $.
Such $\mathbf{D}$ is called 
a \emph{b-$\mathbb{K}$-divisor} (or \emph{b-divisor} for short when $\mathbb{K}$ is clear) %is a possibly infinite linear combination of divisorial valuations $\mathbf{D}=\sum_D d_D D$, such that 
if on each birational model $Y$ of $X$, the trace $\mathbf{D}_Y$ is a $\mathbb{K}$-divisor, or equivalently, $\mathbf{D}_Y$ is a finite sum. 
%If $d_D\neq 0$ in $\mathbf{D}$ for some $D$, $D$ is called a \emph{birational component} of $\mathbf{D}$.

For a $\mathbb{K}$-Cartier divisor $D$ on $X$, the {\it Cartier closure} of $D$ is the b-divisor $\overline{D}$ whose trace on every birational model $f:Y\to X$ is $f^*D$. 
A b-divisor $\mathbf{D}$ is said to be
\emph{b-Cartier} if 
there is a birational model $X'$ over $X$ such that $\mathbf{D}_{X'}$ is Cartier and $\mathbf{D}=\overline{\mathbf{D}_{X'}}$.

When $X$ is projective, a b-divisor $\mathbf{D}$ is said to be
\emph{b-nef} if 
there is a birational model $X'$ over $X$ such that $\mathbf{D}_{X'}$ is $\mathbb{K}$-Cartier and nef, and $\mathbf{D}=\overline{\mathbf{D}_{X'}}$.

\end{defn}

 \subsection{Singularities of g-pairs}\label{sec: gpair}
 
% \han{do not need relative version?do we need sub-pair?}
 \begin{defn}
A \emph{generalized pair} ({\it g-pair} for short) $(X,B+\bM)$ consists of
\begin{itemize}
 \item a normal projective variety $X$,
 \item an effective $\Rr$-divisor $B$ on $X$ called the {\it boundary part}, and
 \item a b-nef b-$\Rr$-divisor $\bM$ on $X$ called the {\it moduli part},  such that $K_X+B+\bM_X$ is $\Rr$-Cartier.
\end{itemize}
%If $B$ is effective, we call $(X/Z,B+\bM)$ a \emph{generalized pair} (\emph{g-pair} for short). If $Z$ is a point, then we may drop $Z$ from the notation.
\end{defn}

% Let $ (X/Z,B+\bM) $ be a g-(sub-)pair. 
Let $ (X,B+\bM) $ be a g-pair. For a prime divisor $E$ over $X$, take a resolution $f:W \to X$ such that $E$ is a divisor on $W$,
%is a log resolution of $(X,\Supp B)$ on which $\bM$ descends. 
then we may write 
$$ K_{W} + B_W + \bM_W \sim_\Rr f^* ( K_X + B + \bM_X ),$$
where $B_W$ is the unique $ \Rr $-divisor on $W$ such that $f_*B_W=B$. %Let $ E $ be a prime divisor on $ W $. 
The \emph{log discrepancy} of $ E $ with respect to $ (X,B+\bM) $ is defined as 
$$ a(E, X, B+\bM) :=1-\mult_E B_W.$$
%which does not depend on the choice of $W$.
 \begin{defn} Fix a non-negative real number $\epsilon$.
A g-pair $ (X, B+\bM) $ is said to be {\it klt} (resp. {\it $\epsilon$-lc, lc}), if  $a(E, X, B+\bM)>0$ (resp. $\geq \epsilon$, $\geq 0$) for any prime divisor $E$ over $X$. 

An lc g-pair $ (X, B+\bM) $ is said to be \emph{dlt} if there is a closed subset $V\subset X$ such that
\begin{itemize}
    \item $X\backslash V$ is smooth and $\Supp B|_{X\backslash V}$ is simple normal crossing, and
    \item if $a(E, X, B+{\bf M})=0$ for some prime divisor $E$ over $X$, then $\Center_X(E)|_{X\backslash V}\neq\emptyset$ is an lc center of $(X\backslash V, B|_{X\backslash V})$.
\end{itemize}  \end{defn}
If ${\bf M}=0$, then all definitions coincide with those of a usual pair. In this paper, we only consider (g)-pairs with projective ambient varieties. We refer the reader to \cite{HaconLiu21} and references therein for recent progress on g-pairs.

% \subsection{Canonical bundle formula}

\subsection{Contractions of Fano type}
\begin{defn}[{\cite{PS09}}]\label{defn RelFano}
Let $\pi: X\to Z$ be a contraction between normal varieties, $X$ is said to be \emph{of Fano type} over $Z$ if one of the following equivalent conditions holds:
\begin{enumerate}
 \item there exists a klt pair $(X,B)$ such that $-(K_X+B)$ is ample over $Z$;
% \item there exists a klt pair $(X,B')$ such that $-(K_X+B')$ is nef and big over $Z$;
 \item there exists a klt pair $(X,B')$ such that $K_X+B'\equiv_Z 0$ and $-K_X$ is big over $Z$;
  \item there exists a klt g-pair $(X,B''+\bM)$ such that $K_X+B''+\bM_X\equiv_Z 0$ and $-K_X$ is big over $Z$.
\end{enumerate}
Here for the equivalence of (3) and (1), we can use the proof of \cite[Lemma~3.24]{Bir18}.
%When $Z$ is a point, we just say that $X$ is of Fano type.
\end{defn}

It is well-known that if $X\to Z$ is a contraction of Fano type, then we can run the MMP for any $\mathbb{R}$-Cartier divisor on $X$ over $Z$, and the property of being Fano type is preserved by MMP and contractions, see for example \cite[\S 2.10]{Bir19} for details.

 %CHECK to here 3.12
 \subsection{Log Calabi--Yau fibrations}
 \begin{defn}
 A pair $(X, B)$ is called a {\it log Calabi--Yau pair} ({\it logCY pair} for short) if $K_X+B\sim_{\mathbb{R}} 0$.
 A g-pair $(X, B+{\bf M})$ is called a {\it log Calabi--Yau g-pair} ({\it logCY g-pair} for short) if $K_X+B+{\bf M}_X\sim_{\mathbb{R}} 0$. 
 \end{defn}

 \begin{defn}\label{defn cy fibration}Fix positive integers $d, r$ and a positive real number $\epsilon$. 
 A {\it weak $(d, r, \epsilon)$-logCY fibration} $(X, B+{\bf M})\to (Z, A)$ consists of a g-pair $(X, B+{\bf M})$ and a base-polarized contraction $\pi: X\to (Z, A)$ satisfying the following conditions:
 \begin{enumerate}
 \item $(X, B+{\bf M})$ is $\epsilon$-lc of dimension $d$,
 \item $K_X+B+{\bf M}_X\sim_{\mathbb{R}} \pi^*L$ for some $\mathbb{R}$-Cartier $\mathbb{R}$-divisor $L$ on $Z$,
 %\item $l{\bf M}$ is b-Cartier,
 \item $A^{\dim Z}\leq r$, and
 \item $A-L$ is pseudo-effective.
 \end{enumerate}

 We call $(X, B+{\bf M})\to (Z, A)$ a {\it $(d, r, \epsilon)$-logCY fibration} if further $A-L$ is ample.
 
 We say that $(X, B+{\bf M})\to (Z, A)$ is {\it of Fano type} if $\pi:X\to Z$ is of Fano type, or equivalently, $-K_X$ is big over $Z$ (cf. \cite[Lemma~3.24]{Bir18}). %\han{a bit confused}
 
 We say that $(X, B+{\bf M})\to (Z, A)$ is {\it rationally connected} if $\pi: X\to Z$ is rationally connected. 
 \end{defn}
 
 For example, in Theorem~\ref{thm main rel}, $(X, B+{\bf M})\to (Z, A)$ is a $(d, r, \frac{1}{l})$-logCY rationally connected fibration.
 In \cite[Definition~2.1]{Bir18}, a $(d, r, \epsilon)$-logCY fibration of Fano type is called a {\it generalized $(d, r, \epsilon)$-Fano type fibration}.

\begin{lem}\label{lem L+A ample}
Under the settings (1)-(2) of Definition~\ref{defn cy fibration}, $L+d'A$ is ample for all $d'>2d$. In particular, $Z$ is the log canonical model of $(X, B+d'\pi^*A+{\bf M}_X)$. 
\end{lem}

\begin{proof}
%Let $(Z,D+{\bf N} )$ be a g-pair induced by the canonical bundle formula $K_X+B+{\bf M} _{X}\sim_{\Rr}\pi^{*}(K_Z+D+{\bf N}_{Z})$. 
It suffices to show that $L+2dA$ is nef, or equivalently, $K_X+B+\bM_X+2d\pi^*A$ is nef.
This follows immedietely from the length of extremal rays for g-pairs (\cite[Proposition 3.13]{HL18}). Namely, if $K_X+B+\bM_X+2d\pi^*A$ is not nef, then $K_X+B+\bM_X$ is not nef and there exists a curve $C$ such that 
$(K_X+B+\bM_X+2d\pi^*A)\cdot C<0$ and $0>(K_X+B+\bM_X) \cdot C\geq -2d$, which implies that $\pi^*A\cdot C< 1$. 
As $K_X+B+\bM_X\equiv_Z 0$, $C$ is not contracted by $\pi$, so $\pi^*A\cdot C\geq 1$, a contradiction.
\end{proof}

 \subsection{Boundedness}\label{sec: boundedness}
 
 A collection $\mathcal{P}$ of projective varieties is
said to be \emph{bounded} (resp. \emph{bounded in codimension one})
if there exists a projective morphism 
$h\colon \mathcal{X}\rightarrow S$
between schemes of finite type such that
each $X\in \mathcal{P}$ is isomorphic (resp. isomorphic in codimension one) to $\mathcal{X}_s$ 
for some closed point $s\in S$. Here by taking a normalization of $\mathcal{X}$ and applying Noetherian induction, we may assume that each fiber $\mathcal{X}_{s}$ is normal. 

Moreover, if $\mathcal{P}$ is a set of logCY pairs $(X, B)$ (resp. logCY g-pairs $(X, B+{\bf M})$), then the set of $X$ in $\mathcal{P}$ is 
said to be {\it bounded modulo flops} if it is bounded in 
codimension one. %In this case, we have $K_{\mathcal{Z}_s}+f_*^{-1}B\sim_\mathbb{R}0$ (resp. $K_{\mathcal{Z}_s}+f_*^{-1}(B+{\bf M}_X)=K_{\mathcal{Z}_s}+f_*^{-1}B+{\bf M}_{\mathcal{Z}_s}\sim_\mathbb{R}0 $). 

%Note that if $\mathcal{P}$ is a set of klt logCY 
%g-pairs which is bounded modulo flops, and 
%$(X, B+{\bf M})\in \mathcal{P}$ with a small birational
%map $f \colon \mathcal{Z}_s \dashrightarrow X$ as in the definition, 
%then $(\mathcal{Z}_s, f_*^{-1}B+{\bf M})$ is also a klt logCY g-pair by the negativity lemma. 
%Moreover, $(X, B+{\bf M})$ is $\epsilon$-lc if and only 
%if $(\mathcal{Z}_s, f_*^{-1}B+{\bf M})$ is $\epsilon$-lc. 

A collection $\mathcal{P}$ of projective log pairs is said to be \emph{log bounded} if there is a projective morphism 
$h\colon \mathcal{X}\rightarrow S$
between schemes of finite type and a reduced divisor $\mathcal{E}$ on $\mathcal{X}$ where $\mathcal{E}$ does not contain any fiber of $h$, such that for every $(X,B)\in \mathcal{P}$, there is a closed point $s \in S$ and an isomorphism $f \colon \mathcal{X}_s \to X$ such that $\mathcal{E}_s:=\mathcal{E}|_{\mathcal{X}_s}$ coincides with the support of $f_*^{-1}B$.

 %A collection of fibrations between projective varieties $ \mathcal{P}$ is said to be \emph{bounded in codimension one} (resp. \emph{bounded}) if there exist projective morphisms $\mathcal{X} \to \mathcal{Z} \to S$ between schemes of finite type such that for $(X\to Z)\in \mathcal{P}$, there exists a closed point $s\in S$ such that $\mathcal{Z}_s\simeq Z$ and there is a small birational contraction (resp. isomorphism) $\mathcal{X}_s\dashrightarrow X$ commuting with $X\to Z$ and $\mathcal{X}_s\to \mathcal{Z}_s$.

 A collection $\mathcal{P}$ of base-polarized fibrations between projective varieties is
said to be \emph{bounded in codimension one} (resp. \emph{bounded})
if there exist projective morphisms
$\mathcal{X} \to \mathcal{Z} \to S$
between schemes of finite type and a Cartier divisor $\mathcal{A}$ on $\mathcal{Z}$ ample over $S$ such that
for every $(X\to (Z, A) )\in \mathcal{P}$, there exists a closed point $s\in S$ such that $\mathcal{Z}_s\simeq Z$, $\mathcal{A}_s\simeq A$, and there is a small birational map (resp. isomorphism) $\mathcal{X}_s\dashrightarrow X$, where all relations are assumed to commute.

We provide two easy but useful lemmas on boundedness of base-polarized fibrations.

\begin{lem}\label{lem log bdd implies bdd fib}
Let $\mathcal{P}$ be a set of base-polarized fibrations. Suppose that there exists a positive integer $r$ such that for each $\pi: X\to (Z, A)$ in $\mathcal{P}$, there exists a very ample divisor $H$ on $X$ such that $H^{\dim X}\leq r$ and $H-\pi^*A$ is pseudo-effective.
Then $\mathcal{P}$ is bounded.
\end{lem}

\begin{proof}
By the boundedness of Hilbert schemes, we may find 
a projective morphism $h: \mathcal{X} \to S$
between schemes of finite type and a Cartier divisor $\mathcal{L}$ on $\mathcal{X}$ such that for every $(\pi: X\to (Z, A))\in \mathcal{P}$, there exists a point $s\in S$ such that $\mathcal{X}_s\simeq X$ and $\mathcal{L}_s\simeq \pi^*A$. By Noetherian induction and \cite[Corollary 12.9]{GTM52}, after decomposing $S$, we may assume that the stalk of $h_{*}\mathcal{L}$ at $s$ is isomorphic to $H^0(\mathcal{X}_s, \mathcal{L}_s)$ for any $s\in S$. Moreover, since $\pi^*A$ is free, by possibly shrinking $S$, we may assume that $h^*h_*\mathcal{L}\to \mathcal{L}$ is surjective. So $\mathcal{L}$ induces a fibration $\Pi: \mathcal{X}\to \mathcal{Z}$ over $S$ with a Cartier divisor $\mathcal{A}$ on $\mathcal{Z}$ ample over $S$ such that $\Pi^*\mathcal{A}=\mathcal{L}$. Therefore, for $s\in S$ corresponding to $(\pi: X\to (Z, A))\in \mathcal{P}$, $\mathcal{X}_s\to (\mathcal{Z}_s, \mathcal{A}_s)$ is isormorphic to $X\to (Z, A)$, which shows that $\mathcal{P}$ is bounded.
\end{proof}

\begin{lem}\label{lem bounded chain}
Let $\mathcal{P}$ and $\mathcal{P}'$ be two sets of base-polarized fibrations. Denote $\mathcal{P}''$ to be the set of base-polarized fibrations $X\to (Z, H_Z)$ such that $X\to Z$ is induced by $X\to Y\to Z$ for some $(f: X\to (Y, H_Y))\in \mathcal{P}$ and $(g: Y\to (Z, H_Z))\in \mathcal{P'}$ with $H_Y-g^*H_Z$ pseudo-effective. 
If $\mathcal{P}'$ is bounded and $\mathcal{P}$ is bounded in codimension one (resp. bounded),
then $\mathcal{P}''$ is bounded in codimension one (resp. bounded).
\end{lem}

\begin{proof}
After replacing $X$ by a birational model, it suffices to prove the case that $\mathcal{P}$ is bounded.
In this case,
there exists a positive integer $r$ such that for every $(f: X\to (Y, H_Y))\in \mathcal{P}$, there exists a very ample divisor $H_X$ on $X$ such that $H_X^{\dim X}\leq r$ and $H_X-f^*H_Y$ is ample. In particular, $H_X-f^*g^*H_Z$ is pseudo-effective. So
%$f^*g^*H_Z\cdot H_X^{\dim X-1}\leq r$ and hence $(X, f^*g^*H_Z)$ belongs to a log bounded family. Hence
$\mathcal{P}''$ is bounded by Lemma~\ref{lem log bdd implies bdd fib}.
\end{proof}

 \section{Proofs of main theorems}
 
\subsection{Boundedness results}
 
The following theorem is a reformulation of {\cite[Theorem~2.2]{Bir18}}. % which shows that $X$ belongs to a bounded family. %is a special case of \cite[Theorem~2.2]{Bir18} when $L=0$. % and a bounded set of varieties $\mathcal{P}$.
 \begin{thm}[{\cite[Theorem~2.2]{Bir18}}]\label{thm: bir18 special}
 Fix positive integers $d, r$ and a positive real number $\epsilon$. Then the set of base-polarized fibrations $X\to (Z, A)$ such that there exists a $(d, r, \epsilon)$-logCY fibration of Fano type $(X, B+{\bf M})\to (Z, A)$ as in Definition~\ref{defn cy fibration},
 is bounded. 
 \end{thm}
 
 \begin{proof}
 %In {\cite[Theorem~2.2]{Bir18}}, it is showed that $X$ belongs to a bounded family. We explain that in fact the proof implies that $X\to (Z, A)$ belongs to a bounded family.
 
%Keep the notation in Definition~\ref{defn cy fibration}. % then we may assume that $A$ is effective. 
By applying {\cite[Theorem~2.3]{Bir18}} to $P=\pi^*A$, there exists a positive real number $t$ depending only on $d, r$, and $\epsilon$ such that 
$(X, B+t\pi^{*}A+\bM)$ is klt where $B+t\pi^{*}A$ is the boundary part. By assumption, $A-L$ is ample. Here we may assume that $t<1$, so $2A-\frac{t}{2}A-L$ is ample.
Now $(X, B+\frac{t}{2}\pi^{*}A+\bM)\to (Z, 2A)$ is a $(d, 2^d r, \frac{\epsilon}{2})$-logCY fibration of Fano type. By applying {\cite[Theorem~2.2]{Bir18}} to $\Delta=\frac{t}{2}\pi^{*}A$, $(X, \frac{t}{2}\pi^{*}A)$ belongs to a log bounded family. Since coefficients of $t\pi^{*}A$ are less than $1$,  we have $\frac{t}{2}\pi^{*}A\leq \Supp \pi^{*}A$, and hence there exists a positive integer $r'$ depending only on $d, r$ and $\epsilon$, and a very ample divisor $H$ on $X$ such that $H^{d}\leq r'$ and $H-\pi^*A$ is ample.
Hence the conclusion follows from Lemma~\ref{lem log bdd implies bdd fib}.
 \end{proof}
The following lemma is useful in order to show that if a set of base-polarized fibrations is bounded in codimension one, then the set of certain birational models of them remains bounded in codimension one. %The proof is well-known to experts (cf. \cite[Theorem~5.1]{CDCHJS21}, \cite[Theorem~6.1]{Jia21}).
 
 \begin{lem}\label{bdd contraction}
 Fix a positive rational number $\epsilon$ and positive integers $d, r, l$.
 Let $\mathcal{P}$ be a set of weak $(d, r, \epsilon)$-logCY fibrations $(Y, C)\to (Z, A)$ with $lC$ integral and $\bM=0$.
 Suppose that $\mathcal{P}$ is bounded in codimension one. Denote $\mathcal{P}'$ to be 
 the set consisting of all weak $(d, r, \epsilon)$-logCY fibrations $(X, B)\to (Z, A)$ such that
 \begin{itemize}
 \item there exists a birational morphism $h: Y\to X$ with $h^*(K_X+B)=K_Y+C$,
 \item all $h$-exceptional divisors are in $\Supp C$, and
 \item $((Y, C)\to (Z, A))\in \mathcal{P}$ where the fibration $Y\to Z$ is induced by $Y\to X\to Z$.
 \end{itemize}
 Then the set of  base-polarized fibrations $X\to (Z, A)$ in $\mathcal{P}'$ is bounded in codimension one.
 \end{lem}
 \begin{proof}
As $\mathcal{P}$ is bounded in codimension one, there are quasi-projective schemes $\mathcal{Y}, \mathcal{Z}$, a Cartier divisor $\mathcal{A}$ on $\mathcal{Z}$ ample over $S$, 
and projective morphisms $\mathcal{Y}\overset{\mu}{\rightarrow} \mathcal{Z}\to S$, where $S$ is a disjoint union of finitely many varieties %and $\mathcal{C}_i$ does not contain any fiber, 
such that
for every $((Y, C) \to (Z, A))\in \mathcal{P}$, there is a closed point $s \in S$ and a small birational
map $f_s : {\mathcal{Y}_{s}} \dashrightarrow Y$ such that ${\mathcal{Y}_{s}}$ is normal, $\mathcal{Z}_s\simeq Z$, and $\mathcal{A}_s\simeq A$.
%such that $\mathcal{C}_{i,s}=f_*^{-1}\Supp C$. 
We may assume that the set of points $s$ corresponding to $((Y, C)\to (Z, A))\in \mathcal{P}$ is dense in $S$. After decomposing $S$, we may also assume that $S$ is smooth and affine.

For a point $s$ corresponding to $((Y,C)\to (Z, A))\in \mathcal{P}$, as $f_s$ is small and $\mathcal{Y}_s$ is normal, %$$K_{{\mathcal{Y}_s}}+f_*^{-1}C\sim_{\mathbb{Q}} f_*^{-1}(K_Y+C)\sim_{\mathbb{Q}} \pi_{{\mathcal{Y}_s}}(L),$$ 
 $\mu_s: ({\mathcal{Y}_s}, f_{s*}^{-1}C)\to (\mathcal{Z}_s,\mathcal{A}_s)$ is a weak $(d, r, \epsilon)$-logCY fibration. So $K_{\mathcal{Y}_s}+f_{s*}^{-1}C\sim_{\mathbb{Q}} \mu_s^*L_s$ for a $\mathbb{Q}$-Cartier $\mathbb{Q}$-divisor $L_s$ on $\mathcal{Z}_s$ and %there exists a very ample Cartier divisor $\mathcal{A}_s$ on $\mathcal{Z}_s$ such that $\mathcal{A}_s^{\dim \mathcal{Z}_s}\leq r$ and
 $\mathcal{A}_s-L_s$ is pseudo-effective. %Moreover, we may assume that $\mathcal{A}_s$ comes from the restriction of an ample Cartier divisor $\mathcal{A}$ on $\mathcal{Z}$.
By the existence of $\mathcal{Y}$, there exists a very ample divisor $H$ on $\mathcal{Y}_s$ and a positive integer $v$ depending only on $\mathcal{P}$ such that $H^d\leq v$, $-K_{\mathcal{Y}_s}\cdot H^{d-1}\leq v$, and $\mu_s^*L_s\cdot H^{d-1}\leq \mu_s^*\mathcal{A}_s\cdot H^{d-1}\leq v$.
As $lf_{s*}^{-1}C$ is integral,
$$
\Supp(f_{s*}^{-1}C) \cdot H^{d-1}\leq lf_{s*}^{-1}C\cdot H^{d-1}=l(-K_{\mathcal{Y}_s}+\mu_s^*L_s)\cdot H^{d-1} \leq 2lv.
$$
In particular, $({\mathcal{Y}_s}, f_{s*}^{-1}C)$ belongs to a log bounded family and the number of components of $\Supp(f_{s*}^{-1}C) $ is at most $2lv$.
Therefore, for a fixed $((Y, C)\to (Z, A))\in \mathcal{P}$, the number of $\mathbb{Q}$-divisors $B_Y$ such that $B_Y=h_*^{-1}B$ for some
$((X, B)\to (Z, A))\in \mathcal{P}'$ corresponding to $(Y, C)\to (Z, A)$ as in the assumption is at most $2^{2lv}$.

So possibly after taking a finite base change of $S$, we may assume that there exist $\mathbb{Q}$-divisors $\mathcal{B}$ and 
$\mathcal{D}$ on $\mathcal{Y}$ which do not contain any fiber of $\mathcal{Y}\to S$ such that for each $((X, B)\to (Z, A))\in \mathcal{P}'$ and $((Y, C)\to (Z, A))\in \mathcal{P}$ as in the assumption, there exists a closed point $s\in S$ and a small birational map $f_s: \mathcal{Y}_s\dashrightarrow Y$ such that
$\mathcal{B}_{s}=f_{s*}^{-1}(h_*^{-1} B)$ and $\mathcal{D}_{s}=f_{s*}^{-1}(C-h_*^{-1} B)$.
Here by assumption, $\Supp \mathcal{D}_{s}$ consists of exactly all divisors on $\mathcal{Y}_s$ which are exceptional over $X$.

Now consider a log resolution $g: \mathcal{Y}'\to \mathcal{Y}$ of $(\mathcal{Y},\mathcal{B}+\mathcal{D})$. Denote by $\mathcal{B}'$ the strict transform of $\mathcal{B}$ and denote by $\mathcal{E}'$ the sum of the strict transform of $\Supp \mathcal{D}$ and all $g$-exceptional prime divisors on $\mathcal{Y}'$. After decomposing $S$ into finitely many locally closed subsets, we may assume that for every $s\in S$, $\mathcal{Y}'_s$ is a log resolution of $(\mathcal{Y}_s,\mathcal{B}_s+\mathcal{D}_s)$ and $\mathcal{E}'_s$ is the sum of the strict transform of $\Supp \mathcal{D}_s$ and all exceptional prime divisors on $\mathcal{Y}'_s$.
%Consider the log pair $(\mathcal{Y}',\mathcal{B}'+(1-\frac{\epsilon}{2})\mathcal{E}')$. 

Fix an integer $d_0>2\dim\mathcal{Y}\geq 2d$.
For a point $s\in S$ corresponding to $((X, B)\to (Z, A))\in \mathcal{P}'$ and $((Y, C)\to (Z, A))\in \mathcal{P}$ as in the assumption, 
%$g_s: \mathcal{Y}'_s\to \mathcal{Y}_s$ is a log resolution of $(\mathcal{Y}_s,\mathcal{B}_s+\mathcal{D}_s)$, and 
we may write
$$
K_{\mathcal{Y}'_s}+C'_s+d_0g_s^*\mu_s^*\mathcal{A}_s:=g_s^*(K_{{\mathcal{Y}_s}}+f_{s*}^{-1}C+d_0\mu_s^*\mathcal{A}_s)\sim_{\mathbb{Q}} g_s^*\mu_s^*(L_s+d_0\mathcal{A}_s)
$$
where the coefficients of $C'_s$ are $\leq 1-\epsilon$ and its support is contained in $\Supp(\mathcal{B}'_s+\mathcal{E}'_s)$. Then
$$
(K_{\mathcal{Y}'}+\mathcal{B}'+(1-\frac{\epsilon}{2})\mathcal{E}'+d_0g^*\mu^*\mathcal{A})|_{{\mathcal{Y}'_s}}
\sim_{\mathbb{Q}} {}g_s^*\mu_s^*(L_s+d_0\mathcal{A}_s)+\mathcal{B}'_s+(1-\frac{\epsilon}{2})\mathcal{E}'_s-C'_s.
$$
Note that $\mathcal{B}'_s+(1-\frac{\epsilon}{2})\mathcal{E}'_s-C'_s\geq 0$ and its support coincides with $\mathcal{E}'_s$ which are precisely the divisors on $\mathcal{Y}'_s$ exceptional over $X$. By Lemma~\ref{lem L+A ample}, $Z_s$ is the log canonical model of 
$(X, B+d_0\pi^*\mathcal{A}_s)$ with $$K_X+B+d_0\pi^*\mathcal{A}_s=\pi^*(L_s+d_0\mathcal{A}_s),$$ where $\pi:X\to \mathcal{Z}_s$ is the natural fibration. So 
$(X, B+d_0\pi^*\mathcal{A}_s)$
is a good minimal model of $({\mathcal{Y}'_s}, \mathcal{B}'_s+ (1-\frac{\epsilon}{2})\mathcal{E}'_s+d_0g_s^*\mu_s^*\mathcal{A}_s)$.
By replacing $\mathcal{A}$ by a general member $\mathcal{A}'$ in its $\mathbb{Q}$-linear sysytem, we may assume that $({\mathcal{Y}'}, \mathcal{B}'+ (1-\frac{\epsilon}{2})\mathcal{E}'+d_0g^*\mu^*\mathcal{A}')$ is log smooth and klt over $S$. 
By \cite[Theorem~1.2]{HMX18}, $({\mathcal{Y}'}, \mathcal{B}'+ (1-\frac{\epsilon}{2})\mathcal{E}'+d_0g^*\mu^*\mathcal{A}')$ has a good minimal model ${\mathcal{Y}}''$ with the log canonical model ${\mathcal{Y}}''\to {\mathcal{Z}''}$ over $S$. In particular, 
${\mathcal{Y}}''\to {\mathcal{Z}''}$
is also a good minimal model with the log canonical model for 
$({\mathcal{Y}'}, \mathcal{B}'+ (1-\frac{\epsilon}{2})\mathcal{E}'+d_0g^*\mu^*\mathcal{A})$ over $S$. We may assume that this good minimal model is obtained by an MMP by \cite[Corollary 2.9]{HX13}.
By the choice of $d_0$ and the the length of extremal rays \cite[Theorem 1]{Kaw length}, this MMP is $g^*\mu^*\mathcal{A}$-trivial. So there is a natural morphism ${\mathcal{Z}''}\to \mathcal{Z}$.
By Noetherian induction, after decomposing $S$ into finitely many locally closed subsets, we may assume that for every $s\in S$, ${\mathcal{Y}}''_s\to \mathcal{Z}''_s$ is a good minimal model with the log canonical model of 
$({\mathcal{Y}'_s}, \mathcal{B}'_s+ (1-\frac{\epsilon}{2})\mathcal{E}'_s+d_0g^*\mu^*\mathcal{A}_s)$.
In particular, for any point $s\in S$ corresponding to $((X, B)\to (Z, A))\in \mathcal{P}'$ and $((Y,C)\to (Z, A))\in \mathcal{P}$ as in the assumption, by
\cite[Theorem~3.5.2]{KM98},
${\mathcal{Y}}''_s$ is isomorphic to $X$ in codimension one as they are both good minimal models of a same pair, and $\mathcal{Z}''_s\simeq \mathcal{Z}_s$ with $\mathcal{A}''_s\simeq\mathcal{A}_s$ where $\mathcal{A}''$ is the pullback of $\mathcal{A}$ on $\mathcal{Z}''$. % and $K_{{\mathcal{Y}}''_s}+\mathcal{B}''_s$ is $\mathbb{Q}$-Cartier where $\mathcal{B}''$ is the strict transform of $\mathcal{B}'$ on $\mathcal{Y}''$. 
Hence the family $\mathcal{Y}''\to (\mathcal{Z}'', \mathcal{A}'')\to S$ shows that $X\to (Z, A)$ is bounded modulo flops.
 \end{proof}

 %%%%%3.25 here
 
 \subsection{Descending logCY g-pairs of fixed index}
 The following proposition is a generalization of \cite[Proposition 6.3]{Bir19} to the setting of generalized pairs.

\begin{prop}[{see \cite[Lemma 4.2]{Has22}, \cite[Theorem~1.5]{FM20}, \cite[Theorem 1.2]{HL21}}]\label{prop:cbfindex}
Let $d,l$ be two positive integers. Then there exists a positive integer $l'$ depending only on $d$ and $l$ satisfying the following.

Assume that $(X,B+\bM)$ is a g-pair and $\pi:X\to Z$ is a contraction such that
\begin{enumerate}
 \item $(X,B+\bM)$ is lc of dimension $d$, and $\dim Z>0$,
 \item $X$ is of Fano type over $Z$, 
 \item $lB$ is integral and $l\bM$ is b-Cartier, and
 \item $K_X+B+\bM_X\sim_{\Qq,Z}0$.
\end{enumerate}
Then there is an lc g-pair 
$(Z,D+\bN)$ such that
$$l'(K_X+B+\bM_X)\sim l'\pi^*(K_Z+D+\bN_Z),$$ and $l'\bN$ is b-Cartier.
Moreover, if $(X, B+\bM)$ is klt, then $(Z,D+\bN)$ is klt.
\end{prop}
%The proof is the same as that of \cite[Proposition 6.3]{Bir19}. , where $D$ and $\bN$ are the discriminant and moduli parts of the canonical bundle formula for $(X,B+\bM)$ over $Z$

As a corollary, we generalize Proposition~\ref{prop:cbfindex} to rationally connected fibrations of finite index instead of fibrations of Fano type.
 
 \begin{cor}\label{cor:cbfindex rc}
Let $d,l$ be two positive integers. Then there exists a positive integer $l'$ depending only on $d$ and $l$ which is divisible by $l$ and satisfies the following.

Assume that $(X,B+\bM)$ is a g-pair and $\pi:X\to Z$ is a contraction such that
\begin{enumerate}
 \item $(X,B+\bM)$ is lc of dimension $d$, and $\dim Z>0$,
 \item $\pi$ is rationally connected, 
 \item $l\bM$ is b-Cartier, and
 \item $l(K_X+B+\bM_X)\sim_{Z}0$.
\end{enumerate}
Then there is an lc g-pair 
$(Z,D+\bN)$ such that
$$l'(K_X+B+\bM_X)\sim l'\pi^*(K_Z+D+\bN_Z),$$ and $l'\bN$ is b-Cartier. Moreover, if $(X, B+\bM)$ is klt, then $(Z,D+\bN)$ is klt.
\end{cor}
 \begin{proof}
 Note that in the assumption $l(K_X+B+\bM_X)$ is Cartier and $l\bM$ is b-Cartier, so $lB$ is automatically integral.

 We prove the statement by induction on $\dim X-\dim Z$. If $\dim X=\dim Z$, then $\pi$ is birational and we may take $D=\pi_{*}B$, ${\bf{N}}={\bf{M}}$, and $l'=l$.
From now on, suppose that $\dim X-\dim Z>0$. By \cite[Proposition 3.9]{HL18}, possibly replacing $(X,B+\bM)$ with its dlt model, we may assume that $(X, B+\bM)$ is $\mathbb{Q}$-factorial dlt. In particular, $X$ is $\mathbb{Q}$-factorial klt.

 \medskip

{\bf Case 1}. $K_X$ is not pseudo-effective over $Z$.

We can run a $K_X$-MMP over $Z$ to get a Mori fiber space $\pi': X'\to Z'$ over $Z$. 
Then by the negativity lemma, $(X', B'+\bM)$ is lc with $l(K_{X'}+B'+\bM_{X'})\sim_{Z}0$. Here $(X', B'+\bM)$ is klt if $(X, B+\bM)$ is klt. 
By applying Proposition~\ref{prop:cbfindex} to $X'\to Z'$, there exists a constant $l''$ depending only on $d,l$ and an lc g-pair $(Z', D'+{\bf N}')$,
such that 
$$l''(K_{X'}+B'+\bM_{X'})\sim l''\pi'^*(K_{Z'}+D'+\bN'_{Z'}),$$ and $l''\bN'$ is b-Cartier. Here $(Z', D'+{\bf N}')$ is klt if $(X', B'+\bM)$ is klt.
We may assume that $l$ divides $l''$, so $l''(K_{Z'}+D'+\bN'_{Z'})\sim_Z 0$. 
Also $Z'\to Z$ is rationally connected as its general fibers are dominated by those of $\pi$. 
So we conclude the statement by applying the inductive hypothesis to $(Z', D'+\bN')\to Z$.

 \medskip

{\bf Case 2}. $K_X$ is pseudo-effective over $Z$.

In this case, %on a general fiber $F$ of $\pi$, $K_F\sim_\mathbb{Q}0$, $B|_F=0$ and $\bM|_F\equiv 0$. Note that $F$ is rationally connected, by \cite[Lemma 6.3]{Jia21}, $F$ is not canonical. Thus
by Lemma~\ref{lem exist E}, there exists a prime divisor ${E_0}$ over $X$ such that $a({E_0}, X, B)<1$ and ${E_0}$ dominates $Z$.
By \cite[Corollary 1.4.3]{BCHM10}, there is a projective birational morphism $h: Y \to X$ extracting only ${E_0}$. Note that $a=a({E_0}, X, B+\bM)= a({E_0}, X, B)<1$ where the equality is from Lemma~\ref{lem exist E}(2), so we may write
\[
K_Y+B_Y+(1-a){E_0}+\bM_Y = h^* (K_X+B+\bM_X),
\]
where $B_Y$ is the strict transform of $B$ on $Y$.
Here $Y\to Z$ is rationally connected as its general fibers are birational to those of $X\to Z$. % by \cite[Corollary 1.4]{HM07}. 
It is clear that $(Y, B_Y+(1-a){E_0}+\bM)\to Z$ satisfies all conditions of Corollary~\ref{cor:cbfindex rc}, and $(Y, B_Y+(1-a){E_0}+\bM)$ is klt if $(X, B+\bM)$ is klt. As $E_0$ dominates $Z$, $K_Y$ is not pseudo-effective over $Z$, and hence we may apply Case 1 to $(Y, B_Y+(1-a){E_0}+\bM)$ to conclude the statement.
 \end{proof}
 
 \begin{lem}\label{lem exist E}
 Let $\pi: X\to Z$ be a rationally connected fibration and $(X, B+\bM)$ a $\mathbb{Q}$-factorial lc g-pair such that $X$ is klt and $K_X+B+\bM_X\equiv_Z 0$. Suppose that $K_X$ is pseudo-effective over $Z$. Then the following statements hold: 
 \begin{enumerate}
 \item for a general fiber $F$ of $Z$, $B_F=0$ and $K_F\equiv \bM_X|_F\equiv 0$;
 
 \item there exists a prime divisor ${E_0}$ over $X$ such that $a({E_0}, X, B)<1$, ${E_0}$ dominates $Z$, and for any birational model $h: Y\to X$ on which ${E_0}$ is a divisor, $\mult_{E_0} (h^*\bM_X-\bM_Y)=0$. %In particular, $a({E_0}, X, B)=a({E_0}, X, B+\bM)$.
 \end{enumerate}
 
%By \cite[Corollary 1.4.3]{BCHM10}, there is a projective birational morphism $\pi: Y \to X$ extracting only ${E_0}$. 
 \end{lem}
 \begin{proof}
 On a general fiber $F$ of $\pi$, $K_F$ and $\bM_X|_F$ are pseudo-effective and $K_F+B|_F+\bM_X|_F\equiv 0$. So $B|_F=0$ and $K_F\equiv \bM_X|_F\equiv 0$. 
 
 Note that $F$ is rationally connected and $K_F\equiv 0$, by \cite[Lemma 6.3]{Jia21}, $F$ is not canonical. Hence by adjunction, $(X, B)$ is not canonical over a non-empty open subset of $Z$. Thus there exists a prime divisor ${E_0}$ over $X$ such that $a({E_0}, X, B)<1$ and ${E_0}$ dominates $Z$. 
 
 To show the last statement, we may assume that $h: Y\to X$ is a sufficiently high model such that $\bM_Y$ is nef. Then by the negativity lemma, 
 $\bM_Y+G=h^*\bM_X$ where $G\ge0 $ is $h$-exceptional.
 Restricting on a general fiber $F_Y$ of $Y\to Z$, we have $\bM_Y|_{F_Y}+G|_{F_Y}=h^*(\bM_X|_F)\equiv 0. $ 
 This implies that $\bM_Y|_{F_Y}\equiv 0$ and $G|_{F_Y}=0$. Therefore, $\mult_{E_0} G=0$ as ${E_0}|_{F_Y}\neq 0$.
 \end{proof}

 \subsection{Lifting morphisms via flops}
 
When we consider a contraction $X\to W$ where $W$ belongs to a family which is bounded modulo flops, we use the following lemma to replace this contraction with a birational modification so that $W$ could be assumed to be in a bounded family. The proof is well-known to experts (cf. \cite[Lemma~3.3]{DCS21}).

\begin{lem}\label{lift flops}
Let $Z$ be a projective variety and
let $X$ and $W$ be normal projective varieties over $Z$. Suppose the following conditions hold:
\begin{enumerate}
    \item $\pi:X\to W$ is a contraction of Fano type and $W$ is $\mathbb{Q}$-factorial;
    \item  $(X, B+{\bf M})$ and $(W, D+{\bf N})$ are klt g-pairs such that $$K_X+B+\bM_X\equiv_Z \pi^*(K_W+D+\bN_W)\equiv_{Z}0;$$
    \item there is a small birational map $\phi_W: W\dashrightarrow W'$ to another normal projective variety $W'$ over $Z$. 
\end{enumerate}
%$(X, B+{\bf M})$ be a klt logCY g-pair and let 
%Let $(X, B+{\bf M})$ and $(W, D+{\bf N})$ be klt g-pairs such that $$K_X+B+\bM_X\equiv \pi^*(K_W+D+\bN_W)\equiv_{Z}0.$$ Suppose that 
Then there exists a $\mathbb{Q}$-factorial projective normal variety $X'$ with a birational contraction $\phi: X\dashrightarrow X'$ and a contraction $\pi': X'\to W'$ of Fano type such that $\pi'\circ\phi=\phi_W\circ\pi$.
%Moreover, if $\pi$ is of Fano type, then $\pi'$ is of Fano type. such that $K_{W'}+(\phi_W)_*D+{\bf N}_{W'}$ is $\mathbb{R}$-Cartier
 \end{lem}
 
\begin{proof}
The following simplified proof is provided by the referee.
 Replacing $X$ by a small $\mathbb{Q}$-factorialization \cite[Corollary~1.4.3]{BCHM10}, we may assume that $X$ is $\mathbb{Q}$-factorial.
 Pick an effective ample divisor $H'$ on $W'$ and denote by $H$ the strict transform of $H'$ on $W$. 
 Then $H$ is $\mathbb{Q}$-Cartier as $W$ is $\mathbb{Q}$-factorial, and it is big as $W$ and $W'$ are isomorphic in codimension one. 

Fix a sufficiently small positive real number $\delta$ such that 
 $(X, B+\delta \pi^*H+{\bf M})$ is klt.  
We claim that $B+\delta \pi^*H+{\bf M}_X$ is big over $Z$. In fact, as $\pi$ is a contraction of Fano type, $B+{\bf M}_X$ is big over $W$. So we may write $B+{\bf M}_X=A_X+E_X$ where $A_X$ is ample over $W$ and $E_X$ is an effective $\mathbb{R}$-divisor on $X$. On the other hand, we may write $H=A_W+E_W$ where $A_W$ is ample over $Z$ and $E_W$ is an effective $\mathbb{R}$-divisor on $W$. Then 
$$
B+\delta \pi^*H+{\bf M}_X=(1-\delta')(B+{\bf M}_X)+(\delta'A_X+\delta \pi^*A_W)+\delta'E_X+\delta \pi^*E_W
$$
is big over $Z$ because $\delta'A_X+\delta \pi^*A_W$ is ample over $Z$ for sufficiently small $\delta'>0$.

Then by \cite[Lemma~4.4]{BZ16}, $(X, B+\delta \pi^*H+{\bf M})$ has a good minimal model $X'$ over $Z$ and $W'$ is just the ample model as $\phi_W$ is small and
$$
K_X+B+\delta \pi^*H+{\bf M}_X\equiv_Z \delta \pi^*H=\delta \pi^*\phi_{W*}^{-1}H'.
$$
Hence $X'\to W'$ are the desired models. Here note that $-K_X'$ is big over $W'$ as $-K_X$ is big over $W$ and $\phi_W$ is birational. 
 \end{proof}

% \begin{lem}\label{lift flops bdd}
%Let $\mathcal{P}$ be a set of logCY g-pairs which is bounded modulo flops. Fix an integer $d$ and a positive real number $\epsilon$.
%Denote $\mathcal{P}'$ to be the set of $(X, B+{\bf M})$ such that 
%\begin{enumerate}
% \item $(X, B+{\bf M})$ is an $\epsilon$-lc logCY g-pair of dimension $d$,
% \item $\pi:X\to W$ is a contraction of Fano type,
% \item $(W, B_W+{\bf M}_W)\in \mathcal{P}$
%\end{enumerate}
%Then $\mathcal{P}'$ is bounded modulo flops.
% \end{lem}

 \subsection{Proof of main theorems}

 \begin{lem}\label{lem 1.4 bir}Under the assumptions in Theorem~\ref{thm main rel},
 if there exists a birational contraction $X\dashrightarrow X'$ over $Z$ such that $X'\to (Z, A)$ is bounded in codimension one, then $X\to (Z, A)$ is bounded in codimension one.
 \end{lem}
 \begin{proof}
 After replacing $X'$ by its birational model, we may assume that $X'\to (Z, A)$ is bounded.
 Then there exists a positive integer $r_0$ independent of $X'$ and a very ample divisor $H'$ on $X'$ such that  $H'^d\leq r_0$ and $H'-\eta^*A$ is ample, where $\eta: X'\to Z$ is the natural morphism.

Denote by $B'$ the strict transform of $B$ on $X'$. Then $l(K_{X'}+B'+\bM_{X'})\sim \eta^* L$. By the negativity lemma,
 for any prime divisor $E$ on $X$ which is exceptional over $X'$, 
 $a(E, X', B'+{\bf M})=a(E, X, B+{\bf M})\leq 1$. 
 So by \cite[Lemma~4.5]{BZ16}, there is a projective birational morphism $g: X''\to X'$ extracting exactly all prime divisors on $X$ which are exceptional over $X'$. In particular, $X''$ is isomorphic to $X$ in codimension one. We may write 
 $$
 K_{X''}+B''+{\bf M}_{X''}=g^*(K_{X'}+B'+{\bf M}_{X'}).
 $$
 Then $(X'', B''+{\bf M})$ is $\frac{1}{l}$-lc and $X''\to X'$ is of Fano type (\cite[\S 2.13(7)]{Bir19}).
 So $(X'', B''+{\bf M})\to (X', H')$ is a $(d, r_0, \frac{1}{l})$-log CY fibration of Fano type.
 By applying Theorem~\ref{thm: bir18 special}, $X''\to (X', H')$ belongs to a bounded family.
Then $X''\to (Z, A)$ belongs to a bounded family by Lemma~\ref{lem bounded chain} and the construction of $H'$, which shows that $X\to (Z, A)$ is bounded in codimension one. 
 \end{proof}

 \begin{proof}[Proof of Theorem~\ref{thm main rel}]
We prove the statement by induction on $\dim X-\dim Z$. If $\dim X=\dim Z$, then $\pi$ is birational and hence $-K_X$ is big over $Z$. So the statement follows from Theorem~\ref{thm: bir18 special}. Now suppose that $\dim X>\dim Z$.

 Replacing $X$ by a small $\mathbb{Q}$-factorialization by \cite[Lemma~4.5]{BZ16}, we may assume that $X$ is $\mathbb{Q}$-factorial. In particular, $X$ is klt.
 
 \medskip

 %Now come back to the proof of Theorem~\ref{thm main rel}.
 
{\bf Case 1.} $K_X$ is not pseudo-effective over $Z$. %First we treat the case that either $B\neq 0$ or ${\bf M}_X\not\equiv 0$.

 In this case, we can run a $K_X$-MMP over $Z$ to get a Mori fiber space $X'\to Z'$ over $Z$ with fibrations $\pi': X'\to Z'$ and $f:Z'\to Z$. Here $Z'$ is $\mathbb{Q}$-factorial as $X'$ is  $\mathbb{Q}$-factorial.
 As $l(K_X+B+{\bf M}_X)\sim \pi^*L$, $l(K_{X'}+B'+{\bf M}_{X'})\sim \pi'^*f^*L $ where $B'$ the strict transform of $B$ on $X'$,
  and $(X', B'+\bM)$ is klt by the negativity lemma.
%\begin{align}
 %p^*(K_X+B+{\bf M}_X)=q^*(K_{X'}+B'+{\bf M}_{X'})\label{eq XX'}
 %\end{align}
 %for a common resolution $p: W\to X$, $q: W\to X'$.
 
 $$\xymatrix@R=2em{
X\ar@{.>}[rr] \ar[rdd]_{\pi} & & X' \ar[d]_{\pi'}\\
& & Z'\ar[dl]_{f}\\
&Z &}
$$

 If $\dim Z'=\dim Z=0$, then certainly $Z'\to (Z, A)$ is bounded. If $\dim Z'>0$, then by Corollary~\ref{cor:cbfindex rc}, there exists a positive integer $l'$ depending only on $d, l$ such that there exists a klt g-pair $(Z', D'+{\bf N}')$ with
 $l'(K_{Z'}+D'+{\bf N}'_{Z'})\sim \frac{l'}{l}f^*L$ and $l'{\bf N}'$ b-Cartier. Also $\dim Z'<\dim X$ and $Z'\to Z$ is rationally connected. Hence by induction on dimension, we may assume that $Z'\to (Z, A)$ is bounded in codimension one. By Lemma~\ref{lift flops}, after replacing $X'$ and $Z'$ by their birational models, we may assume that 
 \begin{enumerate}
 \item $X\dashrightarrow X'$ is a birational contraction,
 \item $\pi': X'\to Z'$ is of Fano type, and
 \item $f: Z'\to (Z, A)$ belongs to a bounded family.
 \end{enumerate}
% Hence $(X', B'+{\bf M})$ is $\frac{1}{l}$-lc. 
Note that after the replacement, it remains true that  $l(K_{X'}+B'+{\bf M}_{X'})\sim \pi'^*f^*L$ and
  $(X', B'+\bM)$ is klt.

 As $Z'\to (Z, A)$ is bounded, there exists a very ample divisor $A'\ge0$ on $Z'$ and a positive integer $r'$ independent of $Z'$ such that $A'^{\dim Z'}\leq r'$ and $A'-f^*A$ is ample. In particular, $A'-\frac{1}{l}f^*L$ is ample and $(X', B'+\bM)\to (Z', A')$ is a $(d, r', \frac{1}{l})$-log CY fibration of Fano type. So
 by Theorem~\ref{thm: bir18 special}, such $X'\to (Z', A')$ belongs to a bounded family. 
 Then $X'\to (Z, A)$ belongs to a bounded family by Lemma~\ref{lem bounded chain} and the construction of $A'$. 
 Therefore, $X\to (Z, A)$
 is bounded in codimension one by Lemma~\ref{lem 1.4 bir}.
 
 %In particular, there exists a very ample divisor $A''\ge0$ on $X'$ and a positive integer $r''$ independent of $X'$ such that $A''^d\leq r''$ and $A''-\pi'^*A'$ is ample. 

 \medskip
 
{\bf Case 2.} $K_X$ is pseudo-effective over $Z$.

%By the negativity lemma, ${\bf M}_{W}\equiv 0$ for any birational model $W\to X$.

In this case, by Lemma~\ref{lem exist E}, $(X, B)$ is a klt pair with $K_F+B|_F\equiv 0$ for a general fiber $F$ of $\pi$. By \cite[Theorem~2.12]{HX13}, $(X, B)$ has a good minimal model ${X'}\to Z'$ over $Z$, where $Z'$ is the log canonical model. Hence $K_{X'}+{B'}\sim_{\mathbb{Q}, Z'} 0$ where ${B'}$ is the strict transform of $B$ on ${X'}$.  Note that $Z'$ is birational to $Z$ as $K_F+B|_F\equiv 0$ for a general fiber $F$ of $\pi$. 
$$\xymatrix@R=2em{
%& & {X'}'\ar[d]_{g}\\
X%\ar@{.>}[rru]^{\text{small}} 
\ar@{.>}[rr] \ar[rdd]_{\pi} & & {X'} \ar[d]_{\pi'}\\
& & Z'\ar[dl]_{f}\\
&Z &}
$$

%For any prime divisor $E$ on $X$ which is exceptional over ${X'}$, $a(E, {X'}, B_{X'}+{\bf M})=a(E, X, B+{\bf M})\leq 1$. So by \cite[Lemma~4.5]{BZ16}, there is a projective birational morphism ${X'}'\to {X'}$ extracting all prime divisors on $X$ which are exceptional over ${X'}$. In particular, ${X'}'$ is isomorphic to $X$ in codimension one. Denote $B_{{X'}'}$ to be the strict transform of $B$ on ${X'}'$.

It is clear that $({X'}, B'+\bM)\to (Z, A)$ satisfies the same conditions as $(X, B+\bM)\to (Z, A)$. 
By Lemma~\ref{lem 1.4 bir}, it suffices to show that $X'\to (Z, A)$ is 
bounded in codimension one. 

Denote the induced fibrations by $\pi': {X'}\to Z'$ and $f: Z'\to Z$.
%So it suffices to show that $W'\to (Z, A)$ is bounded in codimension one.
%After replacing $X$ by $W'$, we may assume further that there exists a morphism $\pi': X\to Z'$ over $Z$ such that $K_X+B\sim_{\mathbb{Q}, Z'} 0$ and $Z'$ is birational to $Z$. 
As $Z'\to Z$ is birational and $\pi$ is rationally connected, $\pi'$ is rationally connected. So by applying Corollary~\ref{cor:cbfindex rc} to $\pi': {X'}\to Z'$, there exists a constant $l'$ depending only on $d,l$ and a klt g-pair $(Z', D'+{\bf N}')$
such that 
$$l'(K_{X'}+{B'}+\bM_{X'})\sim l'\pi'^*(K_{Z'}+D'+\bN'_{Z'})\sim\frac{l'}{l}\pi'^*f^*L,$$ and $l'\bN'$ is b-Cartier. %We may assume that $l$ divides $l'$, so $l'(K_{Z'}+D'+\bN'_{Z'})$ is Cartier. 
In particular, $({Z'}, D'+\bN')\to (Z, A)$ is a $(\dim Z', r, \frac{1}{l'})$-logCY fibration. Moreover, $Z'$ is of Fano type over $Z$ as $-K_{Z'}$ is big over $Z$. 
So by Theorem~\ref{thm: bir18 special}, $Z'\to (Z, A)$ belongs to a bounded family. Then there exists a positive integer $r'$ independent of $Z'$ and a very ample divisor $A'\ge0$ on $Z'$ such that $A'^{\dim Z'}\leq r'$ and $A'-f^*A$ is ample. %, where $f$ is the fibration $Z'\to Z$. 
%By Lemma~\ref{lem bounded chain}, it sufficies to show that $X\to (Z', A')$ is bounded in codimension one.
Therefore, $({X'}, {B'}+\bM)\to (Z', A')$ is a $(d, r', \frac{1}{l})$-logCY fibration. 

Now we have $K_{X'}+{B'}\sim_{\mathbb{Q}, Z'}0$ and $K_{{X'}}+{B'}+\bM_{X'}\sim_{\mathbb{Q}, Z'}0$, so $\bM_{X'}=\pi'^*L'$ where $L'$ is a pseudo-effective $\mathbb{Q}$-divisor on $Z'$.
Then $K_{X'}+{B'}\sim_{\mathbb{Q} }\pi'^*(\frac{1}{l}f^*L-L')$ and $A'-\frac{1}{l}f^*L+L'$ is pseudo-effective. 
In particular, $({X'}, {B'})\to (Z', A')$ is a weak $(d, r', \frac{1}{l})$-logCY fibration.

As ${X'}$ and $X$ are isomorphic over the generic point of $Z$, $K_{X'}$ is pseudo-effective over $Z'$, so by Lemma~\ref{lem exist E}, there exists a prime divisor ${E_0}$ over ${X'}$ such that $a=a({E_0}, {X'}, {B'})<1$ and ${E_0}$ dominates $Z'$.
By \cite[Corollary 1.4.3]{BCHM10}, there is a projective birational morphism $h: Y \to {X'}$ extracting only ${E_0}$. By Lemma~\ref{lem exist E} again, $\bM_Y=h^*\bM_{X'}$, so we may write
\begin{align}
 K_Y +B_Y +(1-a){E_0}{}&= h^* (K_{X'}+{B'});\label{eq 1}\\
 K_Y+B_Y+(1-a){E_0} +\bM_Y {}&= h^* (K_{X'}+{B'}+\bM_{X'}),\label{eq 2}
\end{align}
where $B_Y$ is the strict transform of ${B'}$ on $Y$.
Here $Y\to Z'$ is rationally connected as its general fibers are birational to those of ${X'}\to Z'$. % by \cite[Corollary 1.4]{HM07}. 
So by \eqref{eq 2}, $(Y, B_Y+(1-a){E_0}+\bM)\to (Z', A')$ satisfies all conditions of Theorem~\ref{thm main rel} with integers $d, l, r'$. In particular, $l((1-a){E_0}+B_Y)$ is integral. As ${E_0}$ dominates $Z'$, $K_Y$ is not pseudo-effective over $Z'$. So we may apply Case 1 to $(Y, (1-a){E_0}+B_Y+\bM)\to (Z', A')$ to conclude that $Y\to (Z', A')$ is bounded in codimension one.
On the other hand, by \eqref{eq 1}, $(Y, (1-a){E_0}+B_Y)\to (Z', A')$ is a weak $(d, r', \frac{1}{l})$-logCY fibration with $l((1-a){E_0}+B_Y)$ integral. Hence we may apply Lemma~\ref{bdd contraction} to conclude that ${X'}\to (Z', A')$ is bounded in codimension one.
So by Lemma~\ref{lem bounded chain} and the construction of $A'$, ${X'}\to (Z, A)$ is bounded in codimension one.
 Therefore, $X\to (Z, A)$
 is bounded in codimension one by Lemma~\ref{lem 1.4 bir}.
%After replacing $W$ by its birational model, we conclude that there exists a birational contraction $X\dashrightarrow X'$ such that $X'\to (Z, A)$ is bounded.
 \end{proof}
 
 \begin{proof}[Proof of Theorem~\ref{thm main}]
 This follows from Theorem~\ref{thm main rel} by taking $Z$ to be a point.
 \end{proof}
 \begin{proof}[Proof of Theorem~\ref{thm: main pair rc modulo flop}]
 This follows from Theorem~\ref{thm main} by taking $M=-(K_X+B)$ and ${\bf M}=\overline{M}$.
 \end{proof}
 %Take $M=-(K_X+B)$ and ${\bf M}=\overline{M}$, then $(X, B+{\bf M})$ is a klt g-pair, $l(K_X+B+{\bf M}_X)\sim 0$, and $l{\bf M}$ is b-Cartier. So Theorem~\ref{thm: main pair rc modulo flop} follows directly from Theorem~\ref{thm main}.
 %\bibliographystyle{plain}
 
 %\bibliographystyle{unsrt}

%\medskip

%\noindent\textbf{Data Availability Statement} Data sharing not applicable to this article as no datasets were generated or analysed during the current study.

%\section*{Declarations}\noindent\textbf{Conflict of interest} On behalf of all authors, the corresponding author states that there is no conflict of interest.

\end{document}